\documentclass[11pt]{amsart}

\usepackage{amsmath, amsthm, amssymb, euscript, enumerate}
\usepackage{pxfonts, graphicx, subfigure, epsfig}
\usepackage{mathtools, ulem, appendix, accents, color}

\usepackage{amssymb,mathrsfs,graphicx,enumerate}

\newcommand{\cL}{{\mathcal L}}

\theoremstyle{plain}
\newtheorem{theorem}{Theorem}[section]
\newtheorem{definition}{Definition}[section]

\newtheorem{corollary}[theorem]{Corollary}

\newtheorem{lemma}[theorem]{Lemma}
\newtheorem{proposition}[theorem]{Proposition}

\numberwithin{equation}{section}

\title{Improved energy decay estimate for Dir-stationary $Q$-valued functions and its applications}

\author{Sanghoon Lee}
\address{Department of Mathematics, Princeton University, Princeton, NJ 08544, USA }
\email{sl29@math.princeton.edu }
\begin{document}
%%%%%%%%%%%%%%%%

\date{\today}

%\thanks{\textbf{Acknowledgment.} }

%\tableofcontents
\maketitle
\begin{abstract}
In this paper, we establish an improved decay estimate for the Dirichlet energy of Dir-stationary $Q$-valued functions. As a direct application of this estimate, we derive a Liouville-type theorem for bounded Dir-stationary $Q$-valued functions defined on $\mathbb{R}^m$. Additionally, in an attempt to establish the continuity of Dir-stationary $Q$-valued functions, we confirm that such functions exhibit the Lebesgue property at every point within their domain. Finally, we observe that the improved decay estimate implicates that Dir-stationary $Q$-valued functions reside within a generalized Campanato-Morrey space.

\end{abstract}

\section{Introduction}

The theory of multivalued functions, also known as $Q$-valued functions, plays a fundamental role in Almgren's celebrated paper \cite{A} where he established the regularity theorem for mass-minimizing integral currents. It is considered that solving regularity problems of $Q$-valued Dirichlet-minimizing functions is the first step towards exploring those of mass-minimizing integral currents as they are the appropriate linearized problems.  In light of this perspective, it is natural to commence the investigation of the regularity of stationary currents by examining the regularity properties of Dir-stationary $Q$-valued functions.

Recent extensive research in this area, such as \cite{1, 2, 3, 4, 5, 6}, has been based on the modern revision of Almgren's original theory by De Lellis and Spadaro \cite{DS}. For notations and derivations of basic formulas that we will use in this paper, we direct the reader to \cite[Sections 2 and 3]{DS}. The following is the definition of Dir-stationary $Q$-valued functions.

\begin{definition} \label{def:stationary}
Let $m, n$ be positive integers and $\Omega \subset \mathbb{R}^m$ be an open domain. $ f \in W^{1,2}(\Omega, \mathcal{A}_Q ( \mathbb{R}^n ))$ is a Dir-stationary $Q$-valued function if following properties hold:

(1) (the inner variation formula) For every $\phi \in C_c^{\infty}(\Omega , \mathbb{R}^m)$,  
\begin{equation*}
2\int \sum_i \langle Df_i : Df_i \cdot D\phi \rangle - \int |Df|^2 \mathrm{div} \phi = 0 
\end{equation*}

(2) (the outer variation formula) For every $\psi \in C^{\infty} (\Omega_x \times \mathbb{R}^n , \mathbb{R}^n ) $ s.t. $\mathrm{supp}(\psi) \subset \Omega' \times \mathbb{R}^n$ for some $\Omega' \Subset \Omega$, $|D_u \psi| \le C < \infty$ and $|\psi| + |D_x \psi| \le C(1+|u|)$, 
\begin{equation*}
\int \sum_i \langle Df_i : D_x \psi (x, f_i (x) ) \rangle + \int \sum_i \langle Df_i : D_u \psi (x, f_i(x)) \cdot Df_i(x) \rangle = 0
\end{equation*}
\end{definition}

Lin \cite{L} established the continuity of Dir-stationary $Q$-valued functions for 2-dimensional cases, while more recently, Hirsch and Spolaor \cite{HS} provided a simpler and intrinsic proof for the H\"{o}lder continuity of Dir-stationary $Q$-valued functions in the same dimension. However, for dimensions greater than 2, no prior results exist.

The difficulty in investigating the regularity of Dir-stationary $Q$-valued functions arises from the challenge of improving the energy decay estimate, easily obtained from variational formulas, as shown below:
\begin{equation}\label{eq:basicdecay}
\int_{B_r} |Df|^2 dx \le r^{m-2} \int_{B_1} |Df|^2.
\end{equation}
Unfortunately, this estimate is not sufficiently strong to apply the Campanato-Morrey estimate. In the case of Dir-minimizing $Q$-valued functions, an improved decay rate for $\int_{B_r} |Df|^2 dx$ is proved by constructing suitable competitors for $f$, as shown in the following proposition:
\begin{proposition}\cite[Proposition 3.10]{DS}
Let $f \in W^{1,2} ( B_1 , \mathcal{A}_Q(\mathbb{R}^n)$ be Dir-minimizing. Then, $\int_{B_r} |Df|^2 dx \le r^{m-2+C(m)} \int_{B_1} |Df|^2$. $C(m)$ is a constant depending only on $m$.
\end{proposition}
However, for Dir-stationary $Q$-valued functions, it is not possible to compare the Dirichlet energy of $f$ with that of its competitors. Nonetheless, we have succeeded in improving the estimate (\ref{eq:basicdecay}) in a different form, which is detailed below:

\begin{proposition}\label{lemma:decay'}
Let $f \in W^{1,2}(\Omega, \mathcal{A}_Q(\mathbb{R}^n))$ be a Dir-stationary $Q$-valued function. For any $x \in \Omega$ and $r < \frac{1}{2}\min (\mathrm{dist}(x, \partial \Omega), 1)$, we have 
\begin{equation} \label{improveddecay}
\frac{1}{r^{m-2}} \int_{B_r(x)} |Df|^2 \le \frac{C}{\log \frac{1}{r}}.
\end{equation}
 Here, $C$ depends on $\mathrm{dist}(x, \partial \Omega)$, $m, Q, ||f||_{W^{1,2}(\Omega, \mathcal{A}_Q ( \mathbb{R}^n ))}$.
\end{proposition}

The first direct application of the improved estimate is the Louville-type theorem for bounded Dir-stationary $Q$-valued functions.
\begin{theorem}
Let $f \in W^{1,2}_{\mathrm{loc}}(\mathbb{R}^m, \mathcal{A}_Q(\mathbb{R}^n))$ be a bounded Dir-stationary $Q$-valued function. Then, $f$ is a constant function.
\end{theorem}
It's worth noting that, to the best of the author's knowledge, this method offers a new proof for the classical Liouville theorem for bounded harmonic functions on Euclidean spaces as a specific case. This proof does not rely on the mean-value property.

The second application allows us to establish the Lebesgue property of Dir-stationary $Q$-valued functions, which is a step towards establishing their continuity.
\begin{theorem} \label{thm:main2}
Let $f \in W^{1,2}(\Omega, \mathcal{A}_Q ( \mathbb{R}^n ))$ be a Dir-stationary $Q$-valued function. Then, any $x \in \Omega$ is a Lebesgue point of $f$.
\end{theorem}
This proof is inspired by an algebraic fact stating that the roots of a polynomial depend continuously on its coefficients. As a corollary, we can deduce that the set of discontinuities of a Dir-stationary $Q$-valued function is a meager set.

Finally, we would like to point out that the estimate (\ref{improveddecay}) indicates that Dir-stationary $Q$-valued functions belong to a generalized Campanato-Morrey space.  For more details about generalized Campanato-Morrey spaces, the readers are referred to \cite{Ha} and the references therein. We also observe that an improvement of (\ref{improveddecay}) would yield the continuity of Dir-stationary $Q$-valued functions.

Now, we briefly outline the structure of this paper. In Section 2, we  review the necessary basic materials from \cite{DS} and prove the local boundedness of Dir-stationary $Q$-valued functions. In Section 3, we prove the improved decay estimate Proposition \ref{lemma:decay'}  and use it to prove our main theorems. In Section 4, we introduce generalized Campanato-Morrey spaces, where  Dir-stationary $Q$-valued functions reside, and demonstrate how further enhancement of the estimate established in Section 3 could suffice to prove the continuity of  Dir-stationary $Q$-valued functions.

\section*{Acknowledgement}
The author would like to thank Professor Camillo De Lellis for his encouragement and for introducing the author to this wonderful subject.

\section{Preliminaries}

We review some fundamental facts from \cite{DS}. Follwing the same argument as for Dir-minimizing $Q$-valued functions case, we obtain:

\begin{proposition} \label{lemma:frequency} \cite[Proposition 3.2]{DS}
Let $ f \in W^{1,2}(\Omega, \mathcal{A}_Q ( \mathbb{R}^n ))$ be a Dir-stationary $Q$-valued function and $x \in \Omega$. Then following identities hold for a.e.  $0<r<dist (x, \partial \Omega)$.
\begin{equation} \label{eq:frequency1}
(m-2) \int_{B_r(x)} |Df|^2 = r\int_{\partial B_r(x)} |Df|^2 - 2r \int_{\partial B_r(x)} \sum_i |\partial_{\nu} f_i | ^2
\end{equation}
\begin{equation} \label{eq:frequency2}
\int_{B_r(x)} |Df|^2 = \int_{\partial B_r(x)} \sum_i \langle \partial_{\nu} f_i, f_i \rangle
\end{equation}
where $\nu$ is the outward normal vector.
\end{proposition}

Also, we need the Lipschitz embedding of  $\mathcal{A}_Q(\mathbb{R}^n)$ into an Euclidean space.

\begin{theorem} \cite[Theorem 2.1]{DS} \label{lemma:embedding} 
There exists a set of unit vectors $\Lambda = \{ e_1, \cdots, e_h \} \subset \mathbb{S}^{n-1}$ such that $\xi :  \mathcal{A}_Q(\mathbb{R}^n) \rightarrow \mathbb{R}^{Qh}$ defined as $\xi(T) \coloneqq h^{-1/2} (\pi_1(T), \cdots, \pi_h(T))$ is a Lipschitz embedding. For $T = \sum_i [\![P_i]\!] \in \mathcal{A}_Q(\mathbb{R}^n)$, $\pi_l(T)$ is defined as a $n$-tuple consisted of $P_i \cdot e_l$ rearranged in increasing order.

\end{theorem}

As an initial step in our study, we establish that any Dir-stationary $Q$-valued function $f$ exhibits local boundedness. The subsequent lemma proves an elementary fact concerning subharmonic functions. We can derive the local boundedness of $f$ from the fact that $\sum_i |f_i|^2$ is subharmonic.

\begin{lemma} \label{lemma:mean}
Let $u\in W^{1,1}_{\mathrm{loc}}(\Omega)$, $g\in L_{\mathrm{loc}}^1(\Omega)$, $u, g \ge 0 $. Suppose we have $\Delta u = g$ in the following weak sense : $\int g \eta = - \int Du \cdot D\eta$ for any Lipschitz function $\eta$ with compact support. Then, $\fint_{\partial B_r} u $ and $\fint_{B_r} u$ are non-decreasing with respect to $r$. In particular, $\lim_{r \rightarrow 0} \fint_{B_r} u$ exists uniquely. 
\end{lemma}

\begin{proof}
Denote $|x|$ by $r$. For $R>\rho > 0 $,  define a Lipschitz function $\eta$ as follows:
$$\eta(x) = \eta(r) = \begin{cases}
\frac{1}{\rho^{m-2}} - \frac{1}{R^{m-2}} & r \le \rho \\
\frac{1}{r^{m-2}} - \frac{1}{R^{m-2}} & \rho \le r \le R \\
0 & r \ge R
\end{cases}$$

Testing with $\eta$, we compute:
\begin{align*} 
\int g \eta & = - \int Du \cdot D\eta = -\int _{B_R \setminus B_\rho} D\eta \cdot Du \\
& = \int_{B_R \setminus B_\rho} \Delta \eta u - \int_{\partial B_R} \partial_\nu \eta u + \int_{\partial B_\rho} \partial_\nu \eta u \\
& = C(m)(\fint_{\partial B_R} u - \fint_{\partial B_\rho } u ).
\end{align*}
Here, $C(m) = (m-2) Vol (S_1^{m-1})$. As $g, \eta \ge 0$. we have $\fint_{\partial B_R} u \ge\fint_{\partial B_\rho } u$.
From the following identity,

\begin{align*}
\fint_{B_R} u - \fint_{B_\rho} u & = \frac{1}{m} \big(\frac{\int_0^R r^{m-1} \fint_{\partial B_r} u}{R^m}  -\frac{\int_0^\rho r^{m-1} \fint_{\partial B_r} u}{\rho^m}\big)\\
& = \frac{1}{m} \big[\frac{\int_{\rho}^R r^{m-1} \fint_{\partial B_r} u}{R^m} -  (\frac{1}{\rho^m}-\frac{1}{R^m})\int_0^\rho r^{m-1} \fint_{\partial B_r} u \big].
\end{align*}
we have $\fint_{B_R} u \ge \fint_{B_\rho} u$. Since $u$ is bounded below, the limit $\lim_{r \rightarrow 0} \fint_{B_r} u$ exists uniquely.
\end{proof}

\begin{proposition}(Local Boundedness of $f$) \label{lemma:boundedness}
Let $f = \sum_i [\![f_i]\!]$ be a stationary $Q$-valued function. Denote the $p$-th coordinate of $f_i \in \mathbb{R}^n$ by $f_i^p$ for $1\le p \le n$. Then $\sum_i (f_i^p)^2(x) \le \fint_{B_r(x)}  \sum_i (f_i^p)^2$ for every Lebesgue point $x$ and $r < \mathrm{dist}(x, \partial \Omega)$.
\end{proposition}

\begin{proof}

Let $f_i(x) = (f_i^p(x))_{1 \le p \le n} = (f_i^1(x), \cdots , f_i^n(x))$. In the outer variation formula,  we test with $\psi^q(x) = \eta(x)  \delta_{pq} x^q$, where $\eta$ is a compactly supported Lipshitz function. We have the following identity. 

\begin{equation}
-\int \sum_i D\eta \cdot D(\sum_i (f_i^p)^2) = 2\int \eta \sum_i |Df_i^p|^2.
\end{equation}

In the other words, $\Delta  (\sum_i (f_i^p)^2) =  2\sum_i |Df_i^p|^2$ in the  weak sense. Now the conclusion follows from Lemma \ref{lemma:mean}.

\end{proof}

%%%%%%energy decay estimate

\section{Improved energy decay estimate and its applications}

In this section, we present the improved energy decay estimate for a stationary $Q$-valued function $f$ and utilize it to prove the main theorems. It's important to note that for a stationary $Q$-valued function $f$, the equation (\ref{eq:frequency1}) provides us the estimate
\begin{equation}
\int_{B_r} |Df|^2 dx \le r^{m-2} \int_{B_1} |Df|^2.
\end{equation}
We will enhance the above estimate solely using equations (\ref{eq:frequency1}), (\ref{eq:frequency2}), and the local boundedness of $f$.

\begin{proposition}\label{lemma:decay}
Let $f \in W^{1,2}(\Omega, \mathcal{A}_Q(\mathbb{R}^n))$ be a Dir-stationary $Q$-valued function. For any $x \in \Omega$ and $r < \frac{1}{2}\min (\mathrm{dist}(x, \partial \Omega), 1)$, we have 
\begin{equation}
\frac{1}{r^{m-2}} \int_{B_r(x)} |Df|^2 \le \frac{C}{\log \frac{1}{r}}.
\end{equation}
 Here, $C$ depends on $\mathrm{dist}(x, \partial \Omega)$, $m, Q, ||f||_{W^{1,2}(\Omega, \mathcal{A}_Q ( \mathbb{R}^n ))}$.
\end{proposition}

\begin{proof}
Let $E(r) = \frac{1}{r^{m-2}} \int_{B_r(x)} |Df|^2$. If $E(r_0) = 0 $ for some $r_0$, the inequality we want to prove is trivial.
Let $R = \frac{1}{2}\min (\mathrm{dist}(x, \partial \Omega), 1)$. By Proposition \ref{lemma:boundedness}, there exists $M>0$ such that $|f|\le M$ on the ball $B_R(x)$.

For a.e. $r < R$, we have following inequalities:
\begin{align*}
|\int_{\partial B_r(x) } \sum_i \langle \partial_\nu f_i, f_i  \rangle| &\le \int_{\partial B_r (x) } \sum_i |\partial_\nu f_i | |f_i| \\
& \le M \int_{\partial B_r(x)} \sum_i |\partial_\nu f_i| \\
& \le M \sqrt{Q} \int_{\partial B_r(x)} \sqrt{\sum_i |\partial_\nu f_i|^2} \\
& \le C (\int_{\partial B_r(x)} \sum_i |\partial_\nu f_i|^2 )^{1/2} r^{\frac{m-1}{2}}
\end{align*}
Combining with the equation (\ref{eq:frequency2}), we have
$$\int_{\partial B_r} |\partial_\nu f|^2 \ge \frac{C}{r^{m-1}} (\int_{B_r} |Df|^2)^2.$$ 
Now using the equation (\ref{eq:frequency1}), 
$$r\int_{\partial B_r} |Df|^2 - (m-2)\int_{B_r(x)} |Df|^2 \ge \frac{C}{r^{m-2}} (\int_{B_r} |Df|^2 ) ^2$$
 which gives us the inequality $E'(r) \ge \frac{C}{r} E(r)^2$ or $-\frac{C}{r} \ge -\frac{E'(r)}{E(r)^2}$. 

By integrating the above differential inequality, we get 
\begin{equation}\label{decayeq}
-C \log(\frac{R}{r}) \ge \frac{1}{E(R)} - \frac{1}{E(r)}
\end{equation}
 for $r<R$. Hence, $E(r) \le \frac{1}{C\log (\frac{R}{r}) + \frac{1}{E(R)}}$.

\end{proof}

As a by-product of the above improved estimate, we have the following Liouville-type theorem for bounded Dir-stationary $Q$-valued functions on $\mathbb{R}^m$.
\begin{theorem}
Let $f \in W^{1,2}_{\mathrm{loc}}(\mathbb{R}^m, \mathcal{A}_Q(\mathbb{R}^n))$ be a bounded Dir-stationary $Q$-valued function. Then, $f$ is a constant function.
\end{theorem}
\begin{proof}
Let $E(r) = \frac{1}{r^{m-2}} \int_{B_r(x)} |Df|^2$. Assume there exists $r_0>0$ such that $E(r) > 0$ for all $r\ge r_0$. By following the proof of Proposition \ref{lemma:decay}, we can observe that the inequality (\ref{decayeq}) holds for every $R, r \ge r_0$. 
In particular, we have
$$E(r_0) \le \frac{1}{C\log (\frac{R}{r_0}) + \frac{1}{E(R)}}$$
for all $R \ge r_0$. If we take the limit as $R \rightarrow \infty$, we find that $E(r_0 ) = 0$, which contradicts our assumption.
\end{proof}

With the help of Proposition \ref{lemma:decay} and some algebraic ideas, we are now ready to prove the second main theorem of this article.

\begin{theorem} \label{thm:lebesgue}
Let $f \in W^{1,2}(\Omega, \mathcal{A}_Q(\mathbb{R}^n))$ be a Dir-stationary $Q$-valued function. Then, any $x \in \Omega$ is a Lebesgue point of $f$, or equivalently, $\xi \circ f$.
\end{theorem}

\begin{proof}
Let $r_0 = \frac{1}{2}\min (\mathrm{dist}(x, \partial \Omega), 1)$. By the Poincar\'{e} inequality and Proposition \ref{lemma:decay}, $\fint_{B_r(x)} \big|\xi \circ f - \fint_{B_r(x)} \xi \circ f \big|^2 \le \frac{C}{\log \frac{1}{r}}$ for $r < r_0$. Here, we used the fact that $\xi$ is a Lipschitz embedding. The right hand side converges to 0 as $r \rightarrow 0$. Hence it suffices to prove that the limit $\lim_{r \rightarrow 0} \fint_{B_r(x)} \xi \circ f$ exists uniquely. 

By Proposition \ref{lemma:boundedness}, there exists $M>0$ such that  $|\xi \circ f|<M$ on the ball $B_{r_0}(x)$. Suppose there exists $r_j \rightarrow 0$ and $r_j' \rightarrow 0$ s.t. $\lim_j \fint_{B_{r_j}(x)} \xi \circ f = \beta$, and $\lim_j \fint_{B_{r_j'}(x)} \xi \circ f = \beta'$. We will prove below that $\beta = \beta'$ must hold. This will confirm the assertion of the theorem. For the sake of brevity, we will abbreviate $B_r(x)$ as simply $B_r$.

We first note that by the Cauchy-Schwarz inequality, we have
\begin{align*}
\lim_j \fint_{B_{r_j}} \big|(\xi \circ f)(y) - \beta \big|dy & \le \lim_j  \fint_{B_{r_j}} \big| (\xi \circ f)(y) - \fint_{B_{r_j}} \xi \circ f \big| + \big|\fint_{B_{r_j}} \xi \circ f - \beta \big|\\
&\le \lim_j \big( \fint_{B_{r_j}} \big| (\xi \circ f)(y) - \fint_{B_{r_j}} \xi \circ f \big|^2 \big)^{\frac{1}{2}} = 0
\end{align*}
from which the equality $\lim_{j} \fint_{B_{r_{j}}} \big| (\xi \circ f)(y) - \beta \big| dy= 0$ holds. Likewise, we have $\lim_{j'} \fint_{B_{r_{j'}}} \big| (\xi \circ f)(y) - \beta' \big| dy= 0$.

From the inequality:
$$\mathrm{dist}(\beta, Im(\xi)) \le \lim_{j} \fint_{B_{r_j}} \big|\xi \circ f - \beta \big|^2 = \lim_{j} \fint_{B_{r_j}} \big|\xi \circ f - \fint_{B_{r_j}} \xi \circ f \big|^2 = 0,$$
we see that there exist $\alpha = [\![\alpha_i]\!], \alpha'= [\![\alpha_i' ]\!] \in \mathcal{A}_{Q} (\mathbb{R}^n)$ such that $\beta = \xi (\alpha)$, $\beta' = \xi (\alpha')$.

Let's choose $e_h \in \Lambda$. Our aim is to prove that  $\{ \alpha_i \cdot e_h \} \equiv \{ \alpha'_i \cdot e_h \}$, or in other words, $\sum_i (\alpha_i \cdot e_h)^k = \sum_i (\alpha'_i \cdot e_h)^k$ holds true for every positive integer $k$. Without loss of generality, we assume $e_h = e_1 = (1, 0, \cdots , 0)$ and the first $Q$ components of $\xi \cdot f$ are a rearrangement of $\{ \alpha_i \cdot e_h \}$ in the increasing order. Since the Dirichlet energy is isotropic, our proof is applicable to any $e_h \in \Lambda$ in the exact same manner.

For $k=1$, $\sum_i f_i (x)$ is a harmonic function by \cite[Lemma 3.23]{DS}. It is trivial to see that  $\lim_{r \rightarrow 0} \fint_{B_r} \sum_i f_i (x)$ is uniquely determined. Hence, the equality $\sum_i \alpha_i = \sum_i \alpha_i'$ holds as vectors. In particular, $\sum_i \alpha_i \cdot e_h = \sum_i \alpha_i' \cdot e_h$.

First, we prove that $ \lim_j \fint_{B_{r_j}} \big|\sum_{a=1}^Q [(\xi \circ f)^a]^k(y) - \sum_{a=1}^Q (\beta^a)^k \big| = 0$. The superscript $a$ denotes $a$-th component of vectors in the Euclidean space. The proof is an easy application of the triangle inequality: 
\begin{align*}
 \lim_j \fint_{B_{r_j}} \big|\sum_{a=1}^Q [(\xi \circ f)^a]^k(y) - \sum_{a=1}^Q (\beta^a)^k \big| & \le  \lim_j \fint_{B_{r_j}} \sum_{a=1}^Q \big| [(\xi \circ f)^a]^k(y) - (\beta^a)^k \big| \\
& \le k M^{k-1} \lim_j \fint_{B_{r_j}} \sum_{a=1}^Q \big| (\xi \circ f)^a(y) - (\beta^a) \big|  \\
& = 0.
\end{align*}
By the same argument, $\lim_j \fint_{B_{r_j'}} \big|\sum_{a=1}^Q [(\xi \circ f)^a]^k(y) - \sum_{a=1}^Q (\beta'^a)^k \big| = 0$.

Next, we prove that actually $\sum_{a=1}^Q (\beta^a)^k = \sum_{a=1}^Q (\beta'^a)^k$ holds. Note that we can translate $f$ by the vector $(M, \cdots, M)$ so that the translated function is still a stationary $Q$-valued function and $f_i \cdot e_1 \ge 0$ for $1 \le i \le Q$. In the outer variation formula, we test with $\psi^q (x) = \eta (x) \delta_{1q}(x^q)^{k-1}$ where $\eta$ is a compactly supported Lipshictz function. Then we have the following identity:

\begin{equation}\label{eq:kweak}
-\int \sum_i D\eta \cdot D(\sum_i (f_i^1)^k) = k(k-1) \int \eta \sum_i |Df_i^1|^2 (f_i^1)^{k-2}.
\end{equation}
In other words, $\Delta (\sum_i (f_i^1)^k ) = k(k-1)\sum_i |Df_i^1|^2 (f_i^1)^{k-2}$ in the weak sense. As $f_i^1 \ge 0$, we may apply Lemma \ref{lemma:mean}, to see that  the limit of $\fint_{B_r}  \sum_i (f_i^1)^k$, or $\fint_{B_r} \sum_{a=1}^Q [(\xi \circ f)^a]^k $  exists uniquely. 

By the Poincar\'{e} inequality and Proposition \ref{lemma:decay},
\begin{align*}
& \lim_{r \rightarrow 0} \fint_{B_{r}} \big|\sum_{a=1}^Q [(\xi \circ f)^a]^k(y) - \fint_{B_r} \sum_{a=1}^Q [(\xi \circ f)^a]^k\big| \\
& \le  \lim_{r \rightarrow 0} \big(  \fint_{B_{r}} \big|\sum_{a=1}^Q [(\xi \circ f)^a]^k(y) - \fint_{B_r} \sum_{a=1}^Q [(\xi \circ f)^a]^k\big|^2 \big)^{\frac{1}{2}}\\
& \le \lim_{r \rightarrow 0} \big( \frac{C}{r^{m-2}}\int_{B_r} \big|\nabla  (\sum_{a=1}^Q [(\xi \circ f)^a]^k) \big|^2    \big)^{\frac{1}{2}} \\
& \le \lim_{r \rightarrow 0} \big( \frac{CM^{2k-1}}{r^{m-2}}\int_{B_r} |\nabla (\xi \circ f) |^2    \big)^{\frac{1}{2}} \le \lim_{r \rightarrow 0} \big(\frac{CM^{2k-1}}{\log {\frac{1}{r}}}\big)^{\frac{1}{2}} = 0 
\end{align*} 
Therefore, we have established the equality
\begin{equation*}
\sum_{a=1}^Q (\beta^a)^k = \lim_{r \rightarrow 0} \fint_{B_r} \sum_{a=1}^Q [(\xi \circ f)^a]^k = \sum_{a=1}^Q (\beta'^a)^k.
\end{equation*}

This confirms that $\sum_i (\alpha_i \cdot e_h)^k = \sum_i (\alpha'_i \cdot e_h)^k$. As a consequence, the values of elementary symmetric polynomials for $\{ \alpha_i \cdot e_h \}$ and $\{ \alpha'_i \cdot e_h \}$ coincide. Therefore, the sets $\{ \alpha_i \cdot e_h \}$ and $\{ \alpha'_i \cdot e_h \}$ are equivalent. By applying the same argument to any vector $e \in \Lambda$, we conclude that $\alpha = \alpha'$, considering the construction of $\xi$.
\end{proof}

\begin{corollary}
Let $f \in W^{1,2}(\Omega, \mathcal{A}_Q(\mathbb{R}^n))$ be a Dir-stationary $Q$-valued function. Then, the set of discontinuity of $f$ is meager.
\end{corollary}
\begin{proof}
Let's denote $f=\sum_i [\![ f_i ]\!]$. For any $e_h \in \Lambda$, we define $g_{h, k} \coloneqq \sum_i (f_i \cdot e_h)^k$. According to Theorem \ref{thm:lebesgue}, we have $g_{h, k}(x) = \lim_{r\rightarrow 0} \fint_{B_r(x)} g_{h,k}(y)dy$. The expression $\fint_{B_r(x)} g_{h,k}(y)dy$ is a continuous function with respect to $x$ since $f \in W^{1,2}(\Omega, \mathcal{A}_Q(\mathbb{R}^n))$. Therefore, $g_{h, k}(x)$ is the pointwise limit of a sequence of continuous functions.

This leads to the conclusion that $g_{h, k}$ is a Baire function, implying that its set of discontinuity points forms a meager set. Since the union of a finite number of meager sets is itself a meager set, and since the roots of a polynomial continuously depend on its coefficients, we can infer that $f$ is continuous except on a meager set.
\end{proof}

%%%%%% implications of lebesuge property

\section{A connection to generalized Campanato-Morrey spaces}

In this section, we outline a potential approach to establishing the continuity of Dir-stationary $Q$-valued functions. We introduce generalized Campanato-Morrey spaces tailored to our problem.

The improved estimate
\begin{equation}\label{eq:improved}
\int_{B_r(x)} |Df|^2 \le C r^{m-2} \frac{1}{\log \frac{1}{r}}.
\end{equation}
motivates us to define the following generalized Campanato spaces and Morrey spaces.
\begin{definition}
Let $\Omega \subset \mathbb{R}^m$ be a $C^1$-domain. Set $\Omega(x_0, \rho) \coloneqq \Omega \cap B_\rho (x_0)$.
For every $1\le p \le +\infty$, $\lambda, \kappa \ge 0$ define the generalized Morrey space $L^{p, \lambda, \kappa}(\Omega)$

$$L^{p,\lambda, \kappa } (\Omega) \coloneqq \bigl\{ u \in L^p (\Omega ) :  \sup_{x_0 \in \Omega, 0<\rho<\frac{1}{2}} (\log{\frac{1}{\rho}})^\kappa \rho^{-\lambda} \int_{\Omega (x_0, \rho )} |u |^p dx < \infty \bigl\},$$
endowed with the norm defined by
$$||u||^p_{L^{p, \lambda, \kappa}} \coloneqq \sup_{x_0 \in \Omega, 0<\rho<\frac{1}{2}} \rho^{-\lambda}  (\log{\frac{1}{\rho}})^\kappa \int_{\Omega(x_0, \rho)} |u|^p dx$$
and the generalized Campanato space $\mathcal{L}^{p, \lambda, \kappa}(\Omega)$

$$\mathcal{L}^{p,\lambda, \kappa } (\Omega) \coloneqq \bigl\{ u \in L^p (\Omega ) :  \sup_{x_0 \in \Omega, 0<\rho<\frac{1}{2}} (\log{\frac{1}{\rho}})^\kappa \rho^{-\lambda} \int_{\Omega (x_0, \rho )} |u-u_{x_0, \rho} |^p dx < \infty \bigl\}.$$
where $u_{x_0, \rho} \coloneqq \fint_{\Omega(x_0, \rho)} udx$. We give the Campanato space $\cL^{p, \lambda, \kappa}(\Omega)$ the seminorm
$$[u]^p_{p, \lambda, \kappa} \coloneqq \sup_{x_0\in \Omega, 0<\rho <\frac{1}{2}} \rho^{-\lambda}  (\log{\frac{1}{\rho}})^\kappa \int_{\Omega(x_0, \rho)} | u - u_{x_0, \rho}|^p dx$$
and the norm
$$||u||_{\cL_{p, \lambda, \kappa}(\Omega)} \coloneqq [u]_{p, \lambda, \kappa} + ||u||_{L^p(\Omega)}.$$
\end{definition}

In this viewpoint,  Proposition \ref{lemma:decay} is restated as follows.
\begin{proposition}\label{lemma:morrey}
Let $ f \in W^{1,2}(\Omega, \mathcal{A}_Q ( \mathbb{R}^n ))$ be a Dir-stationary $Q$-valued function. Then, $|Df| \in {L}^{2, m-2, 1}(\Omega')$ and $\xi \circ f \in \mathcal{L}^{2, m, 1}(\Omega')$ for every  $\Omega' \Subset \Omega$.
\end{proposition}
\begin{proof}
We apply the Poincar\'e's inequality to the inequality (\ref{eq:improved}).
\end{proof}
Unfortunately, this decay is not enough to give us the continuity of $f$. In general, we need a stronger estimate to prove the continuity.

\begin{proposition}
Suppose $u \in \mathcal{L}^{2,m, \kappa }(\Omega)$ for $\kappa>2$. Then $u$ is  continuous. Moreover, $|u(x)- u(y)| \le C  \frac{[u]_{2, m, \kappa}}{(\log{\frac{1}{|x-y|}})^{\kappa/2-1}}$ for $|x-y|<1/2$.
\end{proposition}

\begin{proof}
For $x_0 \in \Omega$, $0<r<1$, $r_k \coloneqq \frac{r}{2^k}$, we have 
$$\int_{\Omega (x_0, r_k)} |u-u_{x_0, r_k} |^2 dx \le [u]_{2, m, \kappa}^2 r_k^m \frac{1}{(\log{\frac{1}{r_k}})^\kappa}$$
for some $C$ by the definition of $ \mathcal{L}^{2,m, \kappa }(\Omega)$. By the triangle inequality,
\begin{equation}\label{eq:cauchy}
|u_{x_0, r_k} - u_{x_0, r_{k+1}}| \le C \frac{[u]_{2, m, \kappa}}{k^\frac{\kappa}{2}}
\end{equation} 
and for  $k<h$,
\begin{equation}
|u_{x_0, r_k} - u_{x_0, r_{h}}| \le C \frac{[u]_{2, m, \kappa}}{k^{\frac{\kappa}{2}-1}}
\end{equation}
 which implies that $\{u_{x_0, r_k}\}$ is a Cauchy sequence as $\kappa>2$. Set
$$\tilde{u}(x_0) \coloneqq \lim_{h \rightarrow +\infty} u_{x_0, r_{k}}.$$
It is easy to see that this limit doesn't depend on the choice of $r$. From the differentiation theorem of Lebesgue we know that $u_{x, \rho} \rightarrow u(x)$ in $L^1 (\Omega)$, so that $u= \tilde u$ almost everywhere. Taking the limit as $h \rightarrow +\infty$ in (\ref{eq:cauchy}) we conclude
\begin{equation}\label{eq:limit}
|u_{x, r} - u(x)| \le C \frac{[u]_{2, m, \kappa}}{(\log{\frac{1}{r}})^{\kappa/2-1}}.
\end{equation}
from which we see that the convergence of $u_{x, r}$ to $\tilde{u}(x)$ is uniform. As $u_{x, r}$ is continuous with respect to $x$ by the absolute continuity of the Lebesgue integral, $\tilde{u}(x)$ is also continuous by the uniform limit theorem.

Take $x, y \in \Omega$, and $r \coloneqq |x-y|$. Suppose $r<1/2$. We estimate
$$|u(x)- u(y)| \le |u_{x, 2r} - u(x)| + |u_{x, 2r}, - u_{y, 2r}| + |u_{y, 2r} - u(y)|.$$
The first and the third term is estimated by (\ref{eq:limit}), and
\begin{align*}
|u_{x, 2r} - u_{y, 2r}|& \le  \frac{\int_{\Omega(x, 2r)} |u(z)-u_{x, 2r}| dz +\int_{\Omega(y, 2r)} |u(z)-u_{y, 2r}| dz}{|\Omega(x, 2r) \cap \Omega(y, 2R)|}  \\
& \le C (\fint_{\Omega(x, 2r)} |u(z)-u_{x, 2r}| dz +\fint_{\Omega(y, 2r)} |u(z)-u_{y, 2r}| dz )\\
& \le C \frac{[u]_{2, m, \kappa}}{(\log{\frac{1}{r}})^{\kappa/2}}
\end{align*}
by the Cauchy-Schwarz inequality. Combining above inequalities, we get
$$|u(x)- u(y)| \le C  \frac{[u]_{2, m, \kappa}}{(\log{\frac{1}{|x-y|}})^{\kappa/2-1}}$$
for $|x-y|<1/2$.
\end{proof}

We conclude this section by observing that an improvement of Proposition \ref{lemma:morrey} is sufficient to prove the continuity of stationary $Q$-valued functions.
\begin{proposition}
Let $ f \in W^{1,2}(\Omega, \mathcal{A}_Q ( \mathbb{R}^n ))$ be a Dir-stationary $Q$-valued function. Assume $|Df| \in \mathcal{L}^{2,m-2, \kappa }(\Omega)$ for some $\kappa >1$. Then $f$ is continuous in $\Omega$.
\end{proposition}

\begin{proof}
By Proposition \ref{lemma:boundedness}, we may assume $|f|<M$. For a positive integer $k$, let $g_k = \sum_i (f_i^1)^k$ as in the equation (\ref{eq:kweak}). Choose $x_0 \in \Omega$ and $0<r<1$.
As $x_0$ is a Lebesgue point of $g_k$,
\begin{align*}
\fint_{\partial B_r(x_0)} g_k(y)  -g_k(x_0) & = \int_0^r \fint_{\partial B_\rho(x_0)} \partial_\nu g_k \\
& = C(m) \int_0^r \frac{1}{\rho^{m-1}} \int_{B_\rho (x_0)}\Delta g_k \\
& = C(m) \int_0^r \frac{1}{\rho^{m-1}} \int_{B_\rho (x_0)}k(k-1)\sum_i |Df_i^1|^2 (f_i^1)^{k-2} \\
& \le C(m, k) M^{k-2} \int_0^r \frac{1}{\rho^{m-1}} \int_{B_\rho (x_0)}\sum_i |Df_i^1|^2 \\
& \le C(m, k) M^{k-2} [|Df|]_{2, m-2, \kappa} \int_0^r  \frac{1}{\rho (\log{\frac{1}{\rho}})^\kappa}
\end{align*}
$\int_0^r  \frac{1}{\rho (\log{\frac{1}{\rho}})^\kappa}$ is integrable as $\kappa>1$, and $\fint_{\partial B_r(x_0)} g_k(y)$ is continuous as $g_k \in W^{1,2}(\Omega)$.
By the above estimate, $\fint_{\partial B_r(x_0)} g_k(y)$ converges uniformly to $g(x_0)$ as $r \rightarrow 0$. The uniform limit theorem implies that $g_k$ is continuous. 

Therefore, $\sum_i (f_i^1)^k$ is continuous for every $k$. As roots of a polynomial depend continuously on the coefficients, $\{f_i^1 \}_{1\le i \le Q}$ is continuous. From the constuction of $\xi$, $\xi \circ f$ is continuous and so is $f$.

\end{proof}


\begin{thebibliography}{99}

\bibitem{A}  
Frederick J. Almgren, Jr. Almgren’s big regularity paper, volume 1 of World Scientific Monograph Series
in Mathematics. World Scientific Publishing Co. Inc., River Edge, NJ, 2000.

\bibitem{DS} 
Camillo De Lellis and Emanuele Nunzio Spadaro. Q-valued functions revisited. Mem. Amer. Math. Soc., 211(991):vi+79, 2011.

\bibitem{1}
Camillo De Lellis, Andrea Marchese, Emanuele Spadaro, and Daniele Valtorta. Rectifiability and upper Minkowski bounds for singularities of harmonic Q-valued maps. Comment. Math. Helv., 93(4):737–779, 2018.

\bibitem{2}
 Camillo De Lellis and Emanuele Spadaro. Regularity of area minimizing currents I: gradient Lp estimates. Geom. Funct. Anal., 24(6):1831–1884, 2014.

\bibitem{3}
Camillo De Lellis and Emanuele Spadaro. Multiple valued functions and integral currents. Ann. Sc. Norm. Super. Pisa Cl. Sci. (5), 14(4):1239–1269, 2015.

\bibitem{4}
Camillo De Lellis and Emanuele Spadaro. Regularity of area minimizing currents II: center manifold. Ann. of Math. (2), 183(2):499–575, 2016.

\bibitem{5}
Camillo De Lellis and Emanuele Spadaro. Regularity of area minimizing currents III: blow-up. Ann. of Math. (2), 183(2):577–617, 2016.

\bibitem{6}
Emanuele Spadaro, Complex varieties and higher integrability of Dir-minimizing Q-valued functions, Manuscripta Math. 132 (2010), no. 3-4, 415–429.

\bibitem{L} 
Chun-Chi Lin. Interior continuity of two-dimensional weakly stationary-harmonic multiple-valued functions. J. Geom. Anal., 24(3):1547–1582, 2014.


\bibitem{HS}
Jonas Hirsch and Luca Spolaor. Interior regularity for two-dimensional stationary Q-valued maps. arXiv:2211.09052, 2022

\bibitem{Ha} H. Rafeiro, N. Samko, and S. Samko, Morrey-Campanato spaces: an overview, in: Operator Theory, Pseudo-Differential Equations, and Mathematical Physics, edited by Y. I. Karlovich, L. Rodino, B. Silbermann, and L. Rodman, The Vladimir Rabinovich Anniversary Volume, Operator Theory: Advances and Applications Vol. 228 (Birkhäuser, Basel, 2013), pp. 293– 324.

\end{thebibliography}
\end{document}